\documentclass{amsart}

\usepackage{amssymb}
\usepackage{amsmath,amssymb,amsbsy,amsfonts,amsthm,latexsym,amsopn,amstext, amsxtra,euscript,amscd} 

\usepackage{cite}

\usepackage{xy} 
\input xy \xyoption{all}

\newtheorem{theorem}{Theorem}

\newtheorem{lemma}{Lemma}

\newtheorem{corollary}{Corollary}

\theoremstyle{definition}
\newtheorem{definition}{Definition}

\newtheorem{remark}{Remark}

\def\Z{\mathbb Z}
\def\Q{\mathbb Q}
\def\C{\mathbb C}
\def\L{\mathcal L}
\def\H{\mathcal H}
\def\M{\mathcal M}
\def\l{\lambda}
\def\s{\sigma}
\def\a{\alpha}
\def\p{\mathfrak p}
\def\u{\mathfrak u}
\def\v{\mathfrak v}
\def\iso{{\, \cong\, }}
\def\<{\langle}
\def\>{\rangle}
\def\X{\mathcal X}
\def\Aut{\mathrm{Aut}}
\def\bAut{\overline {\mathrm{Aut}}}
\def\emb{\hookrightarrow}

\begin{document}

\title{Hyperelliptic curves of genus 3 with prescribed automorphism group}
\author{J. Gutierrez}

\address{Dept. of Math., Stat. and Comp.     Univ. of Cantabria,     39071 Santander, Spain}
\email{jaime.gutierrez@unican.es}

\thanks{The first author was partially supported by project  MTM2004-07086 of the Spanish Ministry of Science and Technology}

\author{D. Sevilla}
\thanks{The second author was partially supported by Spanish Ministry of Science Project MTM2004-07086.}

\address{Dept. of Comp. Sci. and Software Eng., Concordia University,  1455 de Maisonneuve W., Montreal QC, H3G 1M8 Canada}
\email{sevillad@gmail.com}

\author{T. Shaska}
\thanks{The third author was supported by the NSA grant R1-05-0129}
\address{Dept. of Mathematics, Oakland University, Rochester, MI, 48309-4485.}
\email{shaska@oakland.edu}




\begin{abstract}
We study genus 3 hyperelliptic curves which have an extra involution. The locus $\L_3$ of these
curves is a 3-dimensional subvariety in the genus 3 hyperelliptic moduli $\H_3$. We find a
birational parametrization of this locus by affine 3-space. For every moduli point $\p \in \H_3$
such that $|\Aut (\p)|>2$, the field of moduli is a field of definition. We provide a rational
model of the curve over its field of moduli for all moduli points $\p \in \H_3$ such that
$|\Aut(\p)|>4$. This is the first time that such a rational model of these curves appears in the
literature.
\end{abstract}

\maketitle

\section{Introduction}
Let $\X_g$ be an irreducible, smooth, projective curve of genus $g \geq 3$, defined over the complex field $\C$.
We denote by $\M_g$ the coarse moduli space of smooth curves of genus $g$ and by $\H_g$ the hyperelliptic locus in
$\M_g$. It is well known that $\dim\ \M_g = 3g-3$ and $\H_g$ is a $(2g-1)$ -- dimensional subvariety of $\M_g$. A
curve $\X_g$ is called \emph{bielliptic} if it admits a degree 2 morphism $\pi: \X_g \to E$ onto an elliptic
curve. This morphism is called a \emph{bielliptic structure} in $\X_g$; see \cite{BC}. The locus $\M_g^b$ of
bielliptic curves is a $(2g-2)$ -- dimensional subvariety of $\M_g$.

From the Castelnouvo-Severi inequality it follows that $\X_g$ admits precisely one bielliptic structure for $g
\geq 6$, but if $g \leq 5$ then there are curves which admit more than one bielliptic structure. Since every
bielliptic structure corresponds to an involution in the automorphism group $\Aut(\X_g)$ of the curve, then these
results can be obtained easily if the list of groups that occur as automorphism groups is known for a given genus
$g$. Lately, algorithms have been developed to determine such lists of groups for reasonably small $g$; see
\cite{MS}.

A bielliptic curve of genus $g \geq 4$ can not be hyperelliptic (Castelnouvo-Severi inequality), but for $g=3$
this can be the case. The bielliptic (non-hyperelliptic) curves of genus $3$ were studied in \cite{BC}; see also
\cite{ST}. In this paper we will focus in the hyperelliptic case. Such curves are known in the literature also as
hyperelliptic curves with \emph{extra involutions}. This extends our previous work in \cite{SV1}, \cite{GS},
\cite{SS}.

In the second section we give a brief description of the dihedral invariants and how they are used to describe the
loci of curves with fixed automorphism groups. Such invariants can be helpful in determining the automorphism
group of the curve and determining its field of moduli. For further details on dihedral invariants we refer to
\cite{GS}. The reader can also check \cite{SV1} for the case $g=2$. Further in this section we briefly define the
classical invariants of binary forms.

In section three, we focus on genus 3 hyperelliptic curves. We define the invariants of binary
octavics and compute them explicitly in terms of the coefficients of the curve for curves with
extra involutions. The locus $\L_3$ of genus $3$ hyperelliptic curves with extra involutions is a
3-dimensional subvariety of $\H_3$. Using such explicit expressions for the invariants of the
binary forms we find a birational parametrization of the locus $\L_3$ via dihedral invariants.
Further, we make use of such invariants to study certain subvarieties of the moduli space of
hyperelliptic curves of genus 3. The list of groups that occur as automorphism groups of genus 3
hyperelliptic curves is described. Then, for each group in the list we describe algebraic relations
that define the corresponding locus.

If $\X \in \L_3$ then $V_4 \emb \Aut(\X)$. Let $G$ be a group which occurs as an automorphism group
of hyperelliptic curves of genus 3 and such that $V_4 \emb G$. We describe each locus of curves
with automorphism group $G$ in terms of the dihedral invariants and prove that the field of moduli
of such curves is a field of definition. If $|G|>4$ then a rational model of the curve is provided
over its field of moduli. As far as we are aware, this is the first time that such rational models
are known for genus 3 curves.

\medskip

\noindent \textbf{Notation:} We denote a hyperelliptic curve of genus 3 by $\X_3$. $D_n$ denotes
the dihedral group of order $2n$ and $V_4$ the Klein 4-group. Further, $\Z_n$ denotes the cyclic
group of order $n$.

\section{Dihedral invariants of hyperelliptic curves}
In this section we give a brief review of some of the basic results on hyperelliptic curves with extra involutions
and their dihedral invariants. For details see \cite{GS}.

Let $\X_g$ be a genus $g$ hyperelliptic curve defined over $\C$ and $\Aut(\X_g)$ its automorphism
group. We say that $\X_g$ has an \emph{extra involution} when there is a non-hyperelliptic
involution in $\Aut(\X_g)$. If the fixed field of such an extra involution is an elliptic field
then sometimes these are called \emph{elliptic involutions}. The hyperelliptic involution $\a_0 \in
\Aut(\X_g)$ is in the center of $\Aut(\X_g)$. We denote $\bAut(\X_g) := \Aut(\X_g)/\<\a_0\>$ and
call it the \emph{reduced automorphism group} of $\X_g$.

Let $\X_g$ be a genus $g$ hyperelliptic curve with an extra involution $\a_1 \in \Aut(\X_g)$. Then,
$\X_g$ is isomorphic to a curve given by an equation
$$Y^2=X^{2g+2} + a_{g} X^{2g} + \cdots + a_1 X^2 +1.$$
Such equation is called the \emph{normal equation} of the curve $\X_g$. There is a degree 2 map
$$\phi_1 : \X_g \to C_1$$
where $C_1$ is the hyperelliptic curve with equation
$$Y^2=X^{g+1} + a_{g} X^{g}+ \cdots + a_1 X +1$$
and genus $g_1= \left[ \frac g 2 \right]$. Since every hyperelliptic curve $\X_g$ has the
hyperelliptic involution $\a_0 \in \Aut(\X_g)$, then the extra involutions come in pairs $(\a_1,
\a_2=\a_0 \a_1)$. The extra involution $\a_2$ determines another degree 2 covering
$$\phi_2 : \X_g \to C_2$$
where $C_2$ is the hyperelliptic curve with equation
$$Y^2=X(X^{g+1} + a_{g} X^{g}+ \dots + a_1 X +1)$$
and genus $g_2= \left[ \frac {g+1} 2 \right]$. The curve $\X_g$ is called \emph{bielliptic} if
$C_1$ is an elliptic curve. This always happens if $g=2$ or 3.

The Jacobian $J_{\X_g}$ of $\X_g$ is isogenous to $J_{C_1} \times J_{C_2}$. We say that $J_{\X_g}$
splits. Our goal is to determine the locus of such curves $\X_g$ in the variety of moduli. The
locus of genus $g$ hyperelliptic curves with an extra involution is an irreducible $g$-dimensional
subvariety of $\H_g$ which we denote by $\L_g$. The following
$$u_i := a_1^{g-i+1} \, a_i \, + \, a_g^{g-i+1} \, a_{g-i+1}, \quad 1\leq i\leq g$$
are called \emph{dihedral invariants} of genus $g$. The next theorem shows that $\L_g$ is a rational variety; see
\cite{GS}.
\begin{theorem}\label{thm1}
Let $g\geq 2$ and $(u_1,\dots,u_g)$ be the $g$-tuple of dihedral invariants. Then,
$k(\L_g)=k(u_1,\dots,u_g)$.
\end{theorem}

A generic point $\p \in \L_g$ has automorphism group $\Aut(\p) \iso V_4$. Singular points of $\L_g$
have more than one tuple of dihedral invariants and therefore more than one conjugacy class of
involutions in $\bAut(\p)$. For curves with automorphism group isomorphic to $V_4$ we have the
following:

\begin{corollary}
Let $\X_g$ and $\X_g^\prime$ be genus $g$ hyperelliptic curves with automorphism groups isomorphic
to $V_4$, and $(u_1, \dots ,u_g)$, $(u_1^\prime, \dots , u_g^\prime)$ their respective dihedral
invariants. Then,
$$\X_g \iso \X_g^\prime \ \iff\ (u_1, \dots , u_g)=(u_1^\prime, \dots, u_g^\prime).$$
\end{corollary}

\begin{theorem} \label{thm_u}
Let $\X_g $ be a genus $g$ hyperelliptic curve with an extra involution and $(u_1,\dots,u_g)$ its
corresponding dihedral invariants.

i) If $V_4 \emb \bAut(\X_g)$ then $2^{g-1}\,u_1^2 = u_g^{g+1}$.

ii) Moreover, if $g$ is odd then $V_4 \emb \bAut(\X_g)$ implies that
$$\left(2^r\,u_1 - u_g^{r+1}\right) \, \left(2^r\,u_1 + u_g^{r+1}\right)=0$$
where $r=\frac {g-1} 2$. The first factor corresponds to the case when involutions of $V_4\emb
\bAut(\X_g)$ lift to involutions in $\Aut(\X_g)$. The second factor corresponds to the case when
one of the two involutions lifts to an element of order 4 in $\Aut(\X_g)$.
\end{theorem}
For the proofs of these statements see \cite{GS}.

From a computational point of view, determining the normal equation of a given curve with an extra
automorphism can be done simply by solving a system of equations; this is quite efficient both
theoretically and in practice. If an extra involution $\a=X \to \frac{aX+b}{cX+d}$ is known
explicitly, one can easily find $\s$ such that $\a^\s=X \to -X$, then the equation after applying
$\s^{-1}$ has the form $Y^2=F(X^2)$. One more substitution $X \to \l\,X$ for certain $\l$ will
provide the normal equation.

\subsection{Genus 2 case}
The case $g=2$ has been studied in \cite{SV1}. Every point in $\M_2$ is a triple $(i_1, i_2, i_3)$ of absolute
invariants, see \cite{SV1} for details. The curve of genus 2 with extra involutions has equation
$$Y^2 = X^6+ a_2 X^4 + a_1 X^2+1.$$
We denote its dihedral invariants by
\begin{equation}
\u  := a_1^3 + a_2^3, \, \, \textit{   and    } \, \,  \v  := 2 a_1 a_2.
\end{equation}
The following lemma is proved in [\cite[pg.710]{SV1}.

\begin{lemma}
Let $\X_2$ be a genus 2 curve such that $G:=\Aut(\X_2)$ has an extra involution and $(\u,\v)$ its
dihedral invariants. Then,\\

a) $G\iso \Z_3 \rtimes D_8$ if and only if $(\u,\v)=(0,0)$ or $(\u,\v)=(6750,450)$.

b) $G\iso GL_2(3)$ if and only if $(\u,\v) = (-250,50)$.

c) $G\iso D_{12}$ if and only if $\ \v^2 - 220 \v -16 \u +4500=0$ for $\v \neq 18,\,
140+60\sqrt{5},\, 50$.

d) $G\iso D_8$ if and only if $\ 2\u^2-\v^3=0$, for $\v \neq 2, 18, 0, 50, 450$. Cases $\v =
0,\,450$ and $\u=50$ are reduced to cases a) and b) respectively.
\end{lemma}

Notice that the parameters $u= \displaystyle\frac {\v} 2$ and $v=\u$ instead of $\u, \v$ are used in \cite{SV1}.
The mapping
$$\Phi: (\u,\v) \to (i_1, i_2, i_3)$$
gives a birational parametrization of $\L_2$. The dihedral invariants $\u, \v$ are given explicitly as rational
functions of $i_1, i_2, i_3$, see \cite{SV1}. For $g=2$, the curve $Y^2=X^6-X$ is the only genus 2 curve (up to
isomorphism) which has extra automorphisms and is not in $\L_2$. The automorphism group in this case is $\Z_{10}$.
Relations between elliptic subcovers of such $\X_2$ were studied in detail in \cite{SV1}.

\subsection{Invariants of binary forms}
In this section we define the action of $GL_2(k)$ on binary forms and discuss the basic notions of
their invariants. Let $k[X,Z]$ be the polynomial ring in two variables and let $V_d$ denote the
$(d+1)$-dimensional subspace of $k[X,Z]$ consisting of homogeneous polynomials
$$f(X,Z) = a_0X^d + a_1X^{d-1}Z + ... + a_dZ^d$$
of degree $d$. Elements in $V_d$ are called \emph{binary forms} of degree $d$. We let $GL_2(k)$ act
as a group of automorphisms on $ k[X, Z] $ as follows:
\begin{equation}
\mbox{for} \quad
\begin{pmatrix} a &b \\ c & d \end{pmatrix}
\in GL_2(k)\ , \quad M
\begin{pmatrix} X \\ Z \end{pmatrix} =
\begin{pmatrix} aX+bZ \\ cX+dZ \end{pmatrix}.
\end{equation}
This action of $GL_2(k)$ leaves $V_d$ invariant and acts irreducibly on $V_d$.

Let $A_0,\ A_1,\ldots,A_d$ be coordinate functions on $V_d$. Then the coordinate ring of $V_d$ can
be identified with $k[A_0,\ldots,A_d]$. For any $I \in k[A_0,\ldots,A_d]$ and $M \in GL_2(k)$, we
define $I^M \in k[A_0,\ldots,A_d]$ as follows:
$${I^M}(f):= I(M(f))$$
for all $f \in V_d$. Then this equation and $I^{MN} = (I^{M})^N$ define an action of $GL_2(k)$ on
$k[A_0,\ldots,A_d]$.
A homogeneous polynomial $I\in k[A_0,\dots,A_d,X,Z]$ is called a \emph{covariant} of index $s$ if
$$I^M(f)=\delta^s I(f)$$
where $\delta =\det(M)$. The homogeneous degree in $a_1,\dots,a_n$ is called the \emph{degree} of
$I$, and the homogeneous degree in $X,Z$ is called the \emph{order} of $I$. A covariant of order
zero is called \emph{invariant}. An invariant is a $SL_2(k)$-invariant on $V_d$.

We will use the symbolic method of classical theory to construct covariants of binary forms. Let
$$f(X,Z):=\sum_{i=0}^n \begin{pmatrix} n \\ i \end{pmatrix} a_i X^{n-i} \, Z^i, \qquad
\texttt{}g(X,Z):=\sum_{i=0}^m \begin{pmatrix} m \\ i \end{pmatrix} b_i X^{n-i} \, Z^i$$
be binary forms of degree $n$ and $m$ respectively with coefficients in $k$. We define the
\emph{r-transvection}
\begin{small}
$$(f,g)^r := d \ \sum_{k=0}^r (-1)^k\begin{pmatrix} r \\ k \end{pmatrix} \cdot
\frac {\partial^r f} {\partial X^{r-k}\ \partial Z^k} \cdot \frac {\partial^r g} {\partial X^k\
\partial Z^{r-k} }$$
\end{small}
where
$d=\frac {(m-r)! \, (n-r)!} {n! \, m!}.$
Then $(f,g)^r$ is a homogeneous polynomial in $k[X,Z]$ and therefore a covariant of order $m+n-2r$ and degree 2.
In general, the $r$-transvection of two covariants of order $m,n$ (resp. degree $p, q$) is a covariant of order
$m+n-2r$ (resp. degree $p+q$). See \cite{Al}, \cite{Cl}, \cite{H}, \cite{KS} for details.

\section{Hyperelliptic curves of genus three}
In this section we study hyperelliptic curves of genus 3 with extra involutions. Let $\X_3$ be such
a curve. Then, $\X_3$ has normal equation
$$Y^2 = X^8 + a_3\,X^6 + a_2\,X^4 + a_1\,X^2 + 1,$$
see \cite{GS}. The dihedral invariants of $\X_3$ are
$$u_1=a_1^4+a_3^4 \, \, \quad u_2=(a_1^2+a_3^2)\,a_2 \, \, \quad u_3=2\,a_1\,a_3.$$
If $a_1=a_3=0$, then $u_1=u_2=u_3=0$. In this case $w:=a_2^2$ is invariant. Thus, we define
\begin{equation} \label{def_u}
\begin{split}
\u (\X_3) \, = \left\{ \aligned & w & \mbox{if} \quad a_1=a_3=0,\\
 & (u_1, w, u_3) & \mbox{if} \quad a_1^2+a_3^2=0\ \mbox{and} \, \, a_2\neq 0,\\
 & (u_1, u_2, u_3) & \mbox{otherwise}. \\ \endaligned
\right.
\end{split}
\end{equation}
To have an explicit way of describing a point in the moduli space of hyperelliptic curves of genus 3 we need the
generators of the field of invariants of binary octavics. These invariants are described in terms of covariants of
binary octavics. Such covariants were first constructed by van Gall who showed that the graded ring of covariants
is generated by 70 covariants and explicitly constructed them, see \cite{Shi1}.

Let $f(X,Y)$ be the binary octavic
$$f(X,Y) = \sum_{i=0}^8 a_i X^i Y^{8-i}.$$
We define the following covariants:
\begin{equation}
\begin{split}
 & g=(f,f)^4, \quad k=(f, f )^6, \quad h=(k,k)^2, \quad m=(f,k)^4, \\
 & n=(f,h)^4, \quad p=(g,k)^4, \quad q=(g, h)^4. \\
\end{split}
\end{equation}

Then the following
\begin{equation}\label{J}
\begin{split}
&J_2=(f,f)^8,\quad J_3=(f,g )^8,\quad J_4=(k,k)^4,\quad \\ & J_5=(m,k)^4,\quad J_6 = (k,h )^4,
\quad J_7= (m,h)^4
\end{split}
\end{equation}
are $SL_2(k)$-invariants. Shioda has shown that the ring of invariants is a finitely generated module of
$k[J_2,\dots,J_7]$, see \cite{Shi1}. The expressions of these covariants are very large in terms of the
coefficients of the curve and difficult to compute. However, in terms of the dihedral invariants $u_1,u_2,u_3$
these expressions are smaller. Analogously, $J_{14}$ is the discriminant of the octavic. We define $M:=2u_1+u_3^2$
and assume $M\neq 0$.


\begin{small}

\[ 
\begin{split}
J_2 = & \, \frac 1 M \, \, (560u_1 + 280u_3^2 + 10u_3u_1 + 5u_3^3 + 2u_2^2)\\  \\
J_3 = & \, \frac {u_2 } {a_2 M^2} \, (12u_2^3 + 4200u_1^2 + 4200u_1u_3^2 + 1050 u_3^4 -110u_3u_2u_1-55u_3^3u_2 \\
&+ 7840u_2u_1 + 3920u_2u_3^2) \\  \\
J_4 = &\, \frac {32}{M^2}\,(2u_2^4-1568u_2^2u_1-784u_2^2u_3^2 + 1008u_2u_1^2 + 1008u_2u_1u_3^2+ 252u_2u_3^4  \\
 & + 8u_1^2u_3^2+307328u_1^2 + 307328u_1u_3^2 + 76832u_3^4+62u_3u_2^2u_1+ 31u_3^3u_2^2\\
 &-784u_3u_1^2+8u_1u_3^4+2u_3^6 -784u_3^3u_1-196u_3^5) \\ \\
J_5 = & \, - \frac {16 u_2}{ a_2 M^3}\,(104u_2u_1^2u_3^2-614656u_2u_1u_3^2-41160u_3^6-614656u_2u_1^2- 153664u_2u_3^4\\
& -246960u_1u_3^4-2296u_2^2u_1u_3^2 +104u_2u_1u_3^4-41552u_3u_2u_1^2-41552u_3^3u_2u_1\\
& +26u_2^3u_3u_1+ 1568u_2^3u_3^2 +13u_2^3u_3^3 + 26u_2u_3^6 - 2296u_2^2u_1^2-574u_2^2u_3^4 -10388u_3^5u_2 \\
& + 1120u_3u_1^3 + 840u_3^5u_1 -4u_2^5 -329280u_1^3  + 140u_3^7 -493920u_1^2u_3^2+3136u_2^3u_1 \\
& + 1680u_3^3u_1^2) \\  \\
J_6 &=  \, -\frac {256} {M^3} \,(2u_3u_1 + u_3^3-392u_1-196u_3^2+u_2^2) (-2u_2^4- 8u_1^2u_3^2-8u_1u_3^4-2u_3^6\\
&+ 154u_3u_2^2u_1+77u_3^3u_2^2 + 1568u_2^2u_1 + 784u_2^2u_3^2 + 3024u_2u_1^2 + 3024u_2u_1u_3^2\\
&+ 756u_2u_3^4 + 10192u_3u_1^2 + 2548u_3^5- 307328u_1^2-307328u_1u_3^2 \\
& -76832u_3^4+ 10192u_3^3u_1) \\  \\
J_7 &=  \, \frac {64 u_2} {a_2 M^4}\, (129077760 u_3^6 u_1 - 481890304 u_2 u_1^3 +516311040u_1^3u_3^2 \\
& +387233280u_1^2u_3^4 + 921984u_2^3u_3^4 + 7299040u_3^7u_2-90u_2^5u_3^3  -14896u_2^4u_1^2\\
& -3724u_2^4u_3^4+3360u_3^6u_1^2+1120u_3^8u_1+141120u_2u_1^4+4480u_1^3u_3^4+16134720u_3^8 \\
& + 2086u_3^7u_2^2 + 345u_2^3u_3^6+38u_3^9u_2  + 3687936u_2^3u_1^2 -68600u_3^9 -25480u_2u_3^8 \\
& + 5180672u_2^2u_1^3 + 647584u_2^2u_3^6-1097600u_3u_1^4+ 8u_2^7 + 140u_3^{10} + 258155520u_1^4 \\
 & -1646400u_3^5 u_1^2-548800u_3^7 u_1-9408u_2^5u_1 -4704u_2^5u_3^2 -722835456u_2u_1^2u_3^2\\
 & +16688u_3u_2^2u_1^3 +304u_3^3u_2u_1^3 + 456u_3^5u_2u_1^2  -199920u_2u_1^2u_3^4+7840u_2u_1^3u_3^2\\
 & -14896u_2^4u_1u_3^2+ 25032u_3^3u_2^2u_1^2 +43794240u_3^5u_2u_1 +87588480u_3^3u_2u_1^2 \\
 & + 228u_3^7u_2u_1  +1380u_2^3u_3^4u_1-78400u_2^3u_3u_1^2 + 58392320u_3u_2u_1^3 +1380u_2^3u_3^2u_1^2\\
 & -135240u_2u_1u_3^6 +3885504u_2^2u_1u_3^4+ 3687936u_2^3u_3^2u_1  -60236288u_2u_3^6\\
 & +12516u_3^5u_2^2u_1 -78400u_2^3u_3^3u_1-361417728u_2u_1u_3^4 + 2240u_1^4u_3^2 \\
 & + 7771008u_2^2u_1^2u_3^2-19600u_2^3u_3^5 -180u_2^5u_3u_1-2195200u_3^3 u_1^3 ) \\
 \end{split}
\]

\[
\begin{split}
J_{14} & =  \, \frac {16} {M^4} \, (-1024u_3^4-64u_2^4-4096u_1^2-4096u_1u_3^2 -2304u_2u_1u_3^2 +6u_3^6 + 384u_3^5\\
& + 1024u_2^2u_1+ 512u_2^2u_3^2 -2304u_2u_1^2-576u_2u_3^4 + 456u_1^2u_3^2+ 132u_1u_3^4\\
 &  + 160u_3^3u_2^2 + 1536u_3u_1^2 + 1536u_3^3u_1 -2u_2^2u_1u_3^2 + 320u_3u_2^2u_1 - 144u_3u_2u_1^2 \\
 & -144u_3^3u_2u_1 + 32u_2^3u_1 + 16u_2^3u_3^2-u_2^2u_3^4 - 36u_3^5u_2 + 8u_3^3u_1^2 + 8u_3^5u_1 \\
 & + 432u_1^3+ 2u_3^7 )^2 \\ \\
\end{split}
\]

\end{small}

The next theorem is a direct corollary of Theorem \ref{thm1}. However, we provide a computational
proof.

\begin{theorem}\label{thm_3}
 $k(\L_3)=k(u_1, u_2 , u_3)$.
\end{theorem}

\begin{proof}
Notice that $J_3, J_5, J_7$ have $a_2$ as a factor. However, this does not contradict Theorem
\ref{thm1}. The function field $k(\L_3)$ is generated by absolute invariants. To define such
absolute invariants one must raise $J_3, J_5, J_7$ to some power and therefore absolute invariants
will have only $a_2^2=\frac {2u_2^2} {2u_1+u_3^2}$ as a factor. Hence, computationally we have
shown that $i_1, \dots , i_5 \in k(u_1, u_2, u_3)$, as expected. Let
$$i_j = \frac {p_j (u_1, u_2, u_3)} {q_j (u_1, u_2, u_3)}$$
for $j=1,\dots,5$ and certain polynomials $p_i$, $q_i$. Then, we have the system of equations
$$i_j \cdot {q_j (u_1, u_2, u_3)} - {p_j (u_1, u_2, u_3)} =0, \quad j=1, \dots , 5.$$
We can solve for $u_1, u_2, u_3$ and express them as rational functions in $i_1, i_2, i_3$.
Therefore, $\ k(u_1, u_2, u_3) \, =\, k(\L_3)$.
\end{proof}

\begin{remark}
The expressions of $u_1, u_2, u_3$ as rational functions in $i_1, i_2, i_3$ are rather large and we
don't display them. Using the above equations, one can find explicit equations of the locus $\L_3$
in terms of $i_1,\dots,i_5$. However, such equations are very large and not practical to use.
Instead, using the dihedral invariants $u_1, u_2, u_3$ is much more convenient.
\end{remark}

\subsection{The locus of genus 3 hyperelliptic curve with prescribed automorphism group}
In this section we describe the locus of genus 3 hyperelliptic curves in terms of dihedral invariants or classical
invariants. First we briefly describe the list of groups that occur as automorphism groups of genus 3
hyperelliptic curves. This list has been computed by many authors; we refer to \cite{MS} for the correct result
and a complete list of references.

We denote by $U_6$, $V_8$ the following groups:
$$U_6 := \< x, y \ | \ x^2, y^6, x\, y\, x\, y^4 \>, \qquad
V_8 := \< x, y \ | \ x^4, y^4, \, (x\,y)^2, \, (x^{-1}y)^2 \>.$$
In Table 1 we list the automorphism groups of genus 3 hyperelliptic curves. The first column is the
case number, in the second column the groups which occur as full automorphism groups are given, and
the third column indicates the reduced automorphism group for each case. The dimension $\delta$ of
the locus and the equation of the curve are also given in the next two columns. The last column is
the \textsc{GAP} identity of each group in the library of small groups in \textsc{GAP}.

\begin{table*}[ht!]
\begin{small}
\caption{$\Aut(X_3)$ for hyperelliptic $X_3$}
 \vskip 0.2cm
\begin{center}
\begin{tabular}{||c|c|c|c|c|c||}
\hline \hline & & & & & \\
 & $\Aut(\X_g)$ & $\overline G$ & $\,\delta\,$ & equation $y^2= f(x)$ & Id. \\
 & & & & & \\
\hline \hline & & & & & \\
1 & $\Z_2$ & $\{1\}$ & 5 & $x(x-1)(x^5+ax^4+bx^3+cx^2+dx+e)$ & $(2,1)$ \\
 & & & & & \\
2 & $\Z_2\times\Z_2$ & $\Z_2$ & 3 & $x^8+a_3x^6+a_2x^4+a_1x^2+1$ & $(4,2)$ \\
3 & $\Z_4$ & $\Z_2$ & 2 & $x(x^2-1)(x^4+ax^2+b)$ & $(4,1)$ \\
4 & $\Z_{14}$ & $\Z_7$ & 0 & $x^7-1$ & $(14,2)$ \\
 & & & & & \\
5 & $\Z_2^3$ & $D_4$ & 2 & $(x^4+ax^2+1)(x^4+bx^2+1)$ & $(8,5)$ \\
6 & $\Z_2\times D_8$ & $D_8$ & 1 & $x^8+ax^4+1$ & $(16,11)$ \\
7 & $\Z_2\times \Z_4$ & $D_4$ & 1 & $(x^4-1)(x^4+ax^2+1)$ & $(8,2)$ \\
8 & $D_{12}$ & $D_6$ & 1 & $x(x^6+ax^3+1)$ & $(12,4)$ \\
9 & $U_6$ & $D_{12}$ & 0 & $x(x^6-1)$ & $(24,5)$ \\
10 & $V_8$ & $D_{16}$ & 0 & $x^8-1$ & $(32,9)$ \\
 & & & & & \\
11 & $\Z_2\times S_4$ & $S_4$ & 0 & $x^8+14x^2+1$ & $(48,48)$ \\
 & & & & & \\
\hline\hline
\end{tabular}
\end{center}
\end{small}
\end{table*}

\begin{remark}
Note that $\Z_2$, $\Z_4$ and $\Z_{14}$ are the only groups which don't have extra involutions.
Thus, curves with automorphism group $\Z_2$, $\Z_4$ or $\Z_{14}$ do not belong to the locus $\L_3$.
\end{remark}

We want to describe each of the above subloci and study inclusions among them. In order to study
such inclusions the lattice of the list of groups needs to be determined. The lattice of the groups
for genus 3 is given in Fig. 1. Each group is presented by its \textsc{GAP} identity. Each level
contains cases with the same dimension (i.e., the bottom level correspond to the 0-dimensional
families). The boxed entries correspond to groups with extra involutions.

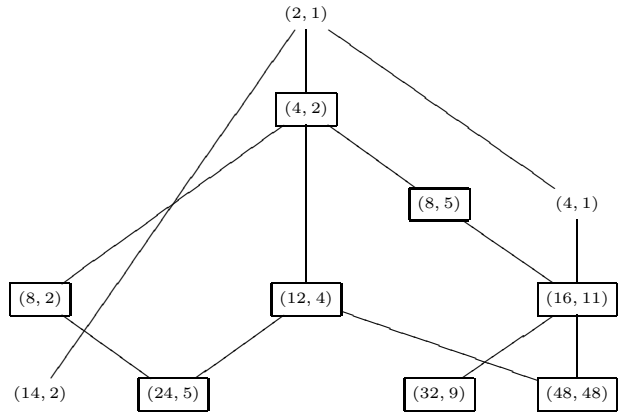
\begin{figure}[ht!]
\begin{tiny}
$$\xymatrix{
  &  &  (2,1) \ar@{-}[d] \ar@{-}[ddrr] \ar@{-}[ddddll]  \\
  &  &  *+[F]{(4,2)} \ar@{-}[dd] \ar@{-}[dr] \ar@{-}[ddll]  \\
  &  &  &  *+[F]{(8,5)} \ar@{-}[dr] &(4,1) \ar@{-}[d]  \\
*+[F]{(8,2)} \ar@{-}[dr]  &  &  *+[F]{(12,4)}\ar@{-}[dl] \ar@{-}[drr]  &  &  *+[F]{(16, 11)}\ar@{-}[dl] \ar@{-}[d]  \\
(14, 2)  &  *+[F]{(24, 5)}  &  &  *+[F]{(32, 9)}  &  *+[F]{(48, 48)}  \\
}$$
\end{tiny}
\caption{The lattice of automorphism groups}
\end{figure}

Next we determine the algebraic relations among dihedral invariants for each case in the diagram.

\begin{theorem}\label{thm_4}
The algebraic relations of dihedral invariants for each case of Figure 1 are presented in Figure 2.
\noindent \begin{tiny}
\begin{figure*}
\hspace{-0.5cm} $ \xymatrix{
  &  &  *+[F]{(u_1, u_2, u_3)} \ar@{-}[dd] \ar@{-}[dr] \ar@{-}[ddl]  \\
  &  &  &  *+[F]{2u_1-u_3^2=0} \ar@{-}[d]  \\
  &  *+[F]{2u_1+u_3^2=0} \ar@{-}[d]  &  *+[F]{Eq.~\eqref{d_u_12}}\ar@{-}[dl] \ar@{-}[dr]  &  *+[F]{a_1=a_3 }\ar@{-}[dl] \ar@{-}[d]  \\
  &  *+[F]{u_2=0}  &  *+[F]{a_1=a_2=a_3=0}  &  *+[F]{(0, 196, 0),\left(\frac {8192} {81}, - \frac {1280} {27}, \frac {128} 9 \right)}  \\
}$
\caption{Corresponding relations of dihedral invariants}
\end{figure*}
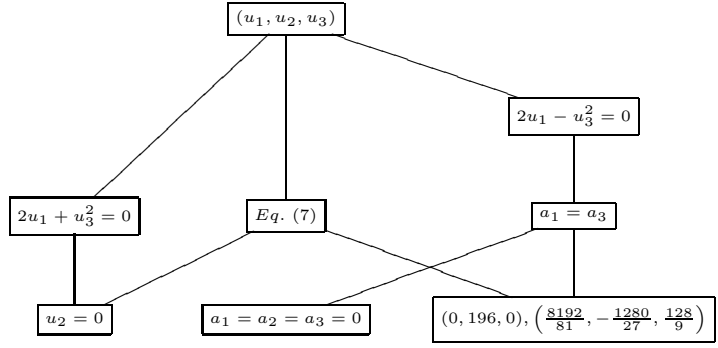
\end{tiny}
\end{theorem}

\begin{proof}
If $\bAut(\X_3)\iso D_4$, then from Theorem \ref{thm_u} we have $4\,u_1^2=u_3^4$. This can be
checked easily by computing the dihedral invariants. If $\,2u_1^2\, = \,u_3^2$ then $\Aut(\X_3)\iso
\Z_2\times \Z_2\times \Z_2$. If $\,2 u_1^2\, = \,-u_3^2$ then $\Aut(\X_3)\iso \Z_2\times \Z_4$, see
the remark after Theorem \ref{thm_u}.

Let $\X_3$ be a curve with equation $Y^2=X\,(X^6 + a X^3 +1)$, where $a\neq 0,\pm 2$. By a
transformation $X \to \frac {X+1} {X-1}$, $\X_3$ has equation
$$Y^2= X^8 + (9\l -7) X^6 + 15 (\l-1) X^4 + (7 \l -9) X^2 -\l$$
where $\l=\frac {a-2}{a+2}$, $\l\neq 0, \pm 1$. Then, by another transformation $X \to
\sqrt[8]{-\l}\,X$, we compute the dihedral invariants:
\begin{small}
\begin{equation}
\begin{split}
u_1 & = - \frac 1 {\l^3} \, (6561\l^6-20412\l^5+26215 \l^4 - 24694 \l^3 \\
    & \, +26215\l^2 -20412\l +6561), \\
u_2 & = 15 \frac {(\l-1)^2 (81\l^2 - 94 \l +81)} {\l^2}, \\
u_3 & = -2 \frac {(7\l -9) (9\l -7)} {\l}.
\end{split}
\end{equation}
\end{small}
Eliminating $\l$, we get $\l = -126 \frac 1 {u_3 - 260}$ and
\begin{equation}\label{d_u_12}
\begin{split}
u_2 & = \frac 5 {588} \, (u_3 -8) (9 u_3 -1024),\\
u_1 & = \frac 9 {2744} u_3^3 - \frac {873} {686} u_3^2 + \frac {149504} {3087} u_3 - \frac
{1048576} {3087}.
\end{split}
\end{equation}
Notice that $u_3\neq 260$, otherwise $a=2$.

If $u_2=0$ then we have $81\l^2-94\l+81=0$ and
$$(u_1, u_2, u_3) \, = \, (8, 0, -32) \mbox{ or } \left( - \frac  {524288} {81}, 0, \frac {1024} 9 \right).$$
Both of these triples satisfy the relation $\, 2u_1\, = \, - u_3^2$ so both $D_{12}$ and
$\Z_2\times\Z_4$ are embedded in $\Aut(\X_3)$. Hence, $\Aut(\X_3)\iso U_6$.

If $2\,u_1^2=u_3^2$ then $81\,\l^2-1568\l+81=0$. Then,
$$(u_1, u_2, u_3)\, = \, (0, 196, 0) \mbox{ or } \left(\frac {8192} {81}, - \frac {1280} {27}, \frac {128} 9\right).$$
Both triples determine the same isomorphism class of curves and correspond to the case
$\Aut(\X_3)\iso\Z_2\times S_4$. The proof is complete.
\end{proof}

The above loci can be easily determined in terms of $J_2,\dots,J_7$. This would be beneficial
because we don't have to find the normal decomposition form of the curve in order to determine the
automorphism group. However, we don't display such equations in terms of $J_2,\dots,J_7$. It is
worth mentioning that, if $\Aut(\X_3)\iso\Z_2\times\Z_4,\ U_6$ or $V_8$, then $J_3=J_5=J_7=0$.

\section{Field of moduli of genus 3 hyperelliptic curves}
In this section we study the field of moduli of genus 3 hyperelliptic curves with extra
automorphisms. Let $\X$ be a curve defined over $\C$. A field $F\subset\C$ is called a \emph{field
of definition} of $\X$ if there exists $\X'$ defined over $F$ such that $\X'$ is isomorphic to $\X$
over $\C$.

\begin{definition}
The \emph{field of moduli} of $\X$ is a subfield $F\subset\C$ such that for every automorphism
$\s\in\Aut(\C)$ the following holds: $\X$ is isomorphic to $\X^\s$ if and only if $\ \s_F = id$.
\end{definition}

We will use $\p=[\X]\in\M_g$ to denote the corresponding \emph{moduli point} and $\M_g(\p)$ the residue field of
$\p$ in $\M_g$. The field of moduli of $\X$ coincides with the residue field $\M_g(\p)$ of the point $\p$ in
$\M_g$. The notation $\M_g(\p)$ (resp. $M(\X)$) will be used to denote the field of moduli of $\p\in\M_g$ (resp.
$\X$). If there is a curve $\X^\prime$ isomorphic to $\X$ and defined over $M(\X)$, we say that $\X$ has a
\emph{rational model over its field of moduli}. As mentioned above, the field of moduli of curves is not
necessarily a field of definition, see \cite{Shi1} for examples of such families of curves.

\begin{lemma}
Let $\u_0\in\L_3(k)$ such that $|\Aut(\u_0)| > 4$. Then, there exists a genus 3 hyperelliptic curve
$\X_3$ defined over $k$ such that $\u(\X_3)=\u_0$ as defined in Eq (\ref{def_u}). Moreover, the
equation of $\X_3$ over its field of moduli is given by:

\medskip

i) If $|\Aut(\X_3)|=16$ then $$Y^2=w X^8+w X^4+1.$$

\smallskip

ii) If $\Aut(\X_3)\iso D_{12}$ then
\begin{small}
\begin{equation}
\begin{split}
Y^2 = & \ (u_3-260) X^8 - 7(u_3-98)X^6 + 15(u_3-134)X^4 \\
  & - 9(u_3-162)X^2 + 126.
\end{split}
\end{equation}
\end{small}

iii) If $\Aut\iso\Z_2\times\Z_4$ then $$Y^2=u_3^4 X^8+u_3^4X^6+8 u_3X^2-16.$$

iv) If $\Aut(\u)\iso\Z_2^3$ then $$Y^2= u_1 X^8 +u_1 X^6 + u_2 X^4 + u_3 X^2 +2.$$
\end{lemma}

\begin{proof}
The proof in all cases consists of simply computing the dihedral invariants. It is easy to check
that these dihedral invariants satisfy the corresponding relations for $\Aut(\X_3)$ given above.
\end{proof}

\begin{corollary}
Let $\p\in\H_3$ such that $|\Aut(\p)|>2$. Then the field of moduli of $\p$ is a field of
definition.
\end{corollary}

\begin{proof}
There is only one hyperelliptic curve of genus 3 which has no extra involutions and whose automorphism group has
more than $4$ elements, see \cite{MS}. This curve is $Y^2= X^7-1$ and its field of moduli is $\Q$. The result
follows from the above Lemma for all groups of order $>4$. Let $\X_3$ be a curve such that its automorphism group
$\Aut(\X_3)$ has order 4. Then $\Aut(\X_3)$ is isomorphic to $\Z_4$ or $V_4$. In both cases the quotient curve
$\X_3 \ \Aut(\X_3)$ is a conic which contains a non-trivial rational point. The result follows from \cite{DE}.
\end{proof}

\begin{remark}
An interesting problem would be to find an algorithm which finds a rational model over the field of moduli for
curves with automorphism group $\Z_2$. In the case of genus 2 this has been done by Mestre; see \cite{Me}.
\end{remark}

For all the computations in this paper we have used Maple; see \cite{Ma}. The above results have been implemented
in a Maple package which is available upon request. In this package we can compute the field of moduli of any
genus 3 hyperelliptic curve and provide a rational model for curves which have more than 4 automorphisms. Further,
the automorphism group can be computed and all invariants of binary octavics defined in Eq.\eqref{J}.

\end{document}